\documentclass[11pt,a4paper]{article}
\usepackage[T1]{fontenc}
\usepackage[utf8]{inputenc}
\usepackage{lmodern}
\usepackage[margin=24mm]{geometry}
\usepackage{amsmath,amssymb,amsthm,mathtools}
\usepackage{booktabs,array,needspace}
\usepackage{microtype}
\usepackage{hyperref}
\hypersetup{colorlinks=true,allcolors=black,
  pdftitle={One Extra Dimension Suffices for the Complete-Graph Squared-Stress},
  pdfauthor={Lilin Yan and Hongwei Zhao}}
\allowdisplaybreaks[2]
\numberwithin{equation}{section}
\newcommand{\funding}[1]{#1}
\newcommand{\email}[1]{\href{mailto:#1}{\nolinkurl{#1}}}
\newtheorem{theorem}{Theorem}[section]
\newtheorem{lemma}[theorem]{Lemma}
\newtheorem{corollary}[theorem]{Corollary}

\theoremstyle{remark}
\newtheorem{remark}[theorem]{Remark}
\newcommand{\R}{\mathbb{R}}
\newcommand{\one}{\mathbf{1}}
\newcommand{\diag}{\operatorname{diag}}
\newcommand{\Diag}{\operatorname{Diag}}

\newcommand{\norm}[1]{\lVert #1\rVert}
\title{One Extra Dimension Suffices for the\\
Complete-Graph Squared-Stress\thanks{\funding{This work was supported by the Key Research and Development Program of Shaanxi Province, China, under grant 2026CY-YBXM-044.}}}
\author{Lilin Yan\thanks{School of Electronics and Information, Northwestern Polytechnical University, Xi'an 710072, China (\email{yanlilin@mail.nwpu.edu.cn}).}
\and Hongwei Zhao\thanks{Corresponding author. School of Electronics and Information, Northwestern Polytechnical University, Xi'an 710072, China (\email{hongvi_zhao@126.com}).}}
\date{}

\begin{document}
\maketitle

\begin{abstract}
Let all pairwise distances of an arbitrary configuration of points in
$\R^\ell$ be known exactly. We prove that every second-order stationary
point of the complete-graph squared-stress objective over configurations
in $\R^k$ has zero residual whenever $k>\ell$. The result holds for every
number of points, without a genericity assumption, and allows repeated
points and degenerate ground-truth configurations. It resolves the
one-extra-dimension conjecture for complete, uniformly weighted data.
The proof combines positivity of the residual stress, polar normalization
of the candidate configuration, a row-wise rank-nullity argument, and
propagation of zero curvature through the full Hessian. The theorem,
including the passage from derivatives of the original quartic objective
to its matrix formulation, has been formalized in Lean~4. The accompanying
verification records include a kernel check with zero trust level and a
transitive axiom audit.
\end{abstract}


\section{Introduction}

Recovering a point configuration from its pairwise distances leads to a
quartic least-squares objective, usually called the squared-stress or
$s$-stress. Even when every distance is known without noise, this
objective can have nonglobal local minima when optimized in the same
dimension as the ground truth. Such examples, and related landscape
results, are established in
\cite{Song2025,CriscitielloMcRaeRebjockBoumal2026}.

A natural relaxation increases the dimension of the candidate points.
Criscitiello, McRae, Rebjock, and Boumal
\cite{CriscitielloMcRaeRebjockBoumal2026} proved complete-data guarantees
in dimensions depending on the number of points, and conjectured that
one additional dimension is sufficient for every ground truth.
Criscitiello~\cite{Criscitiello2026} subsequently established the general
bound $k\ge 2(\ell+1)$, and proved the threshold $k=\ell+1$
when $n\le\ell+3$. The general one-extra-dimension question remained
open in those works.

The present paper proves the conjectured conclusion for all $k>\ell$.
The central argument is local to each row attaining the largest diagonal
entry of the stress matrix. A single scalar constraint is added to a
rank-nullity argument in a polar-normalized space. Positivity forces the
resulting vector to be supported on the maximum-diagonal rows. The full
Hessian then makes the off-diagonal stresses in those rows vanish, which
contradicts the positive diagonal and the zero row sums.

The proof is self-contained. A separate Lean~4 development verifies the
same argument starting from directional derivatives of the original
objective. Section~\ref{sec:formal} describes the precise formal statement,
the correspondence between the mathematical and formal definitions, and
the recorded verification commands.

\section{Problem and main result}\label{sec:statement}

Let $n,k,\ell$ be nonnegative integers and let
$P\in\R^{n\times\ell}$ have rows $p_i^\top$.
For $Z\in\R^{n\times k}$ with rows $z_i^\top$, define
\begin{equation}\label{eq:objective}
 F_P(Z)=\frac12\sum_{1\le i<j\le n}
 \left(\norm{z_i-z_j}^2-\norm{p_i-p_j}^2\right)^2.
\end{equation}
Every unordered pair occurs once and has the same weight. A
\emph{second-order stationary point} satisfies
\begin{equation}\label{eq:sosp}
 \nabla F_P(Z)=0,
 \qquad \nabla^2F_P(Z)\succeq0,
\end{equation}
where the derivatives are taken with respect to all $nk$ coordinates.

\begin{theorem}[One extra dimension]\label{thm:main}
Let $P\in\R^{n\times\ell}$ and $Z\in\R^{n\times k}$ be arbitrary.
If $k>\ell$ and $Z$ is a second-order stationary point of
\eqref{eq:objective}, then
\[
 F_P(Z)=0.
\]
Equivalently,
$\norm{z_i-z_j}^2=\norm{p_i-p_j}^2$ for every pair $i,j$.
Consequently, $Z$ is a global minimizer of $F_P$.
\end{theorem}

No assumptions of distinctness, affine independence, genericity, or full
rank are imposed on $P$. Nor is full column rank assumed for $Z$.
The theorem concerns unrestricted candidate coordinates and exact,
complete, uniformly weighted data.

\begin{corollary}[Local minima and strict saddles]\label{cor:landscape}
Under the hypotheses on $P,n,k,\ell$ in Theorem~\ref{thm:main}, every
local minimizer has zero residual. At every stationary point with
$F_P(Z)>0$, there is a perturbation $U\in\R^{n\times k}$ such that
\[
 \left.\frac{d^2}{dt^2}F_P(Z+tU)\right|_{t=0}<0.
\]
\end{corollary}
\begin{proof}
The polynomial $F_P$ is twice continuously differentiable, so every
local minimizer satisfies \eqref{eq:sosp}. If a stationary point with
positive objective had a positive semidefinite Hessian, it would
contradict Theorem~\ref{thm:main}. A real symmetric matrix that is not
positive semidefinite has a vector of negative quadratic form.
\end{proof}

The corollary is a statement about the landscape. It supplies no uniform
negative-curvature bound, iteration complexity, or convergence guarantee
for a particular optimization method. Translation and rotation
invariance can still produce zero Hessian eigenvalues.

\section{Stress and the full Hessian}\label{sec:stress}

The cases $n=0$ and $n=1$ have identically zero objective. For the rest
of the proof, suppose $n\ge2$. Independent translations of $P$ and $Z$
preserve all pairwise distances. We may therefore work with
\begin{equation}\label{eq:centered}
 \one^\top P=0,\qquad \one^\top Z=0.
\end{equation}
This choice of representatives does not restrict the perturbations in
\eqref{eq:sosp}: the Hessian remains the Hessian in the full coordinate
space. Define
\[
 J=I_n-\frac1n\one\one^\top,\qquad
 G=PP^\top,\qquad Y=ZZ^\top,\qquad C=Z^\top Z.
\]
Let $R$ denote the symmetric residual matrix,
\begin{equation}\label{eq:residual}
 R_{ij}=\norm{z_i-z_j}^2-\norm{p_i-p_j}^2,
 \qquad R_{ii}=0,
\end{equation}
and define the symmetric stress matrix
\begin{equation}\label{eq:stress}
 S=\frac12\bigl(R-\Diag(R\one)\bigr),
 \qquad s_i=S_{ii},
 \qquad D=J\Diag(s)J.
\end{equation}
Here $\diag(A)$ is the vector of diagonal entries of a square matrix,
whereas $\Diag(s)$ is the diagonal matrix with diagonal $s$.
In particular,
\begin{equation}\label{eq:rows}
 S\one=0,\qquad S_{ij}=\frac12R_{ij}\quad(i\ne j),
 \qquad s_i=-\frac12\sum_jR_{ij}.
\end{equation}

\begin{lemma}[Centered Gram identity]\label{lem:gram}
For centered configurations,
\begin{equation}\label{eq:gram}
 G=Y+S-D.
\end{equation}
For every $u\perp\one$,
\begin{equation}\label{eq:diagquadratic}
 u^\top Du=\sum_i s_i u_i^2.
\end{equation}
\end{lemma}
\begin{proof}
Set $E=Y-G$. Expanding squared distances gives
\[
 R=\diag(E)\one^\top+\one\diag(E)^\top-2E.
\]
Since $E$ is centered, $JRJ=-2E$. On the other hand,
\eqref{eq:stress} gives $R=2S-2\Diag(s)$; hence
$JRJ=2S-2D$ by \eqref{eq:rows}. Combining the two identities proves
\eqref{eq:gram}. Equation~\eqref{eq:diagquadratic} follows from $Ju=u$.
\end{proof}

For a perturbation $U$ with rows $u_i^\top$, put
$d_{ij}=z_i-z_j$ and $w_{ij}=u_i-u_j$. The residual on a line is
\[
 R_{ij}(Z+tU)=R_{ij}+2t\,d_{ij}^\top w_{ij}
                    +t^2\norm{w_{ij}}^2.
\]
Differentiating the original objective gives
\begin{align}
 DF_P(Z)[U]
 &=2\sum_{i<j}R_{ij}\,d_{ij}^\top w_{ij},
 \label{eq:firstvariation}\\
 D^2F_P(Z)[U,U]
 &=4\sum_{i<j}(d_{ij}^\top w_{ij})^2
   +2\sum_{i<j}R_{ij}\norm{w_{ij}}^2.
 \label{eq:secondvariation}
\end{align}
Thus
\begin{equation}\label{eq:gradient}
 \nabla F_P(Z)=-4SZ.
\end{equation}
Let
\[
 \mathcal H=\frac14\nabla^2F_P(Z)
\]
and partition it into $k\times k$ blocks indexed by the points.
Expanding \eqref{eq:secondvariation}, using \eqref{eq:centered} and
\eqref{eq:rows}, yields
\begin{align}
 \mathcal H_{ii}&=C+nz_i z_i^\top-s_iI_k,
 \label{eq:diagonalblock}\\
 \mathcal H_{ij}&=-\bigl((z_i-z_j)(z_i-z_j)^\top+S_{ij}I_k\bigr)
 \quad(i\ne j).
 \label{eq:crossblock}
\end{align}
Indeed,
$\sum_j(z_i-z_j)(z_i-z_j)^\top=C+nz_i z_i^\top$,
and $\frac12\sum_jR_{ij}=-s_i$.
At a second-order stationary point we consequently have
\begin{equation}\label{eq:stationaritymatrix}
 SZ=0,\qquad YS=SY=0,\qquad \mathcal H\succeq0.
\end{equation}

We will repeatedly use the elementary implication
\begin{equation}\label{eq:zerocurvature}
 A\succeq0,\quad x^\top Ax=0\quad\Longrightarrow\quad Ax=0.
\end{equation}
For example, it follows by diagonalizing the real symmetric matrix $A$.
The implication concerns the full matrix $A$, including its off-diagonal
blocks.

\section{Stress positivity and a normalized lift}\label{sec:normalization}

\begin{lemma}[Strict diagonal bounds]\label{lem:strictdiag}
Let $A$ be a real symmetric matrix satisfying $A\one=0$.
If $a\ne0$ and $A-aI\succeq0$, then $A_{ii}>a$ for every $i$.
If $a\ne0$ and $aI-A\succeq0$, then $A_{ii}<a$ for every $i$.
\end{lemma}
\begin{proof}
In the first case $A_{ii}\ge a$. Equality would give
$e_i^\top(A-aI)e_i=0$, and \eqref{eq:zerocurvature} would imply
$(A-aI)e_i=0$. Taking the inner product with $\one$ yields
$0=a$, because $\one^\top A=0$, a contradiction. Apply the first
assertion to $-A$ and $-a$ for the second assertion.
\end{proof}

\begin{lemma}[Positivity at a first-order point]\label{lem:stresspsd}
For centered configurations, $SZ=0$ implies $S\succeq0$.
\end{lemma}
\begin{proof}
Suppose the least eigenvalue $a$ of $S$ is negative, and let $v$ be a
corresponding unit eigenvector. Since $S\one=0$ and $a\ne0$,
$v\perp\one$. The identity $Z^\top S=0$ implies $Z^\top v=0$,
and therefore $Yv=0$. Since $S-aI\succeq0$,
Lemma~\ref{lem:strictdiag} gives $s_i>a$ for every $i$. Now
\eqref{eq:gram} and \eqref{eq:diagquadratic} imply
\[
 0\le v^\top Gv
   =a-\sum_i s_i v_i^2
   <a-a\sum_i v_i^2=0.
\]
This contradiction proves $S\succeq0$.
\end{proof}

\begin{lemma}[Consequences of nonzero stress]\label{lem:nonzero}
Suppose \eqref{eq:stationaritymatrix} holds and $S\ne0$. Set
\begin{equation}\label{eq:maxt}
 t=\max_i s_i,\qquad I=\{i:s_i=t\}.
\end{equation}
Then $t>0$, $Z$ has full column rank, and there exist a scalar
$\mu>t$ and a unit vector $v\in\R^n$ such that
\begin{equation}\label{eq:highvector}
 Sv=\mu v,\qquad \one^\top v=0,\qquad Z^\top v=0.
\end{equation}
\end{lemma}
\begin{proof}
By Lemma~\ref{lem:stresspsd}, $S$ is positive semidefinite. A nonzero
positive semidefinite matrix has a positive diagonal entry, so $t>0$.
Choose $i\in I$. If $Za=0$, then $Ca=0$ and $z_i^\top a=0$.
Equation~\eqref{eq:diagonalblock} gives
\[
 0\le a^\top\mathcal H_{ii}a=-t\norm{a}^2,
\]
so $a=0$. Thus $Z$ has full column rank and $C\succ0$.

Let $\mu$ be the largest eigenvalue of $S$. It is positive.
Since $\mu I-S\succeq0$, Lemma~\ref{lem:strictdiag} shows that
$s_i<\mu$ for every $i$, hence $t<\mu$. Choose a unit eigenvector
$v$ for $\mu$. Multiplying $Sv=\mu v$ by $\one^\top$ and by
$Z^\top$ gives the remaining identities in \eqref{eq:highvector}.
\end{proof}

\begin{lemma}[Polar normalization]\label{lem:polar}
Suppose $Z$ is centered and has full column rank. Define
\begin{equation}\label{eq:polar}
 Q=ZC^{-1/2},\qquad C=Z^\top Z\succ0.
\end{equation}
Then
\begin{equation}\label{eq:polaridentities}
 Q^\top Q=I_k,\quad \one^\top Q=0,\quad
 Q^\top YQ=C,\quad Z=QC^{1/2}.
\end{equation}
If $Ca=ta$ with $t>0$, then
\begin{equation}\label{eq:transfer}
 Za=\sqrt t\,Qa.
\end{equation}
If, in addition, $SZ=0$ and $v$ satisfies \eqref{eq:highvector}, then
\begin{equation}\label{eq:orthogonal}
 Q^\top v=0,\qquad SQ=0,\qquad Yv=0.
\end{equation}
\end{lemma}
\begin{proof}
The first, second, and fourth identities in \eqref{eq:polaridentities}
follow from \eqref{eq:polar}. Furthermore,
$Q^\top Z=C^{1/2}=Z^\top Q$, so
$Q^\top YQ=C^{1/2}C^{1/2}=C$.
The positive square root acts by $\sqrt t$ on the $t$-eigenspace of
$C$, proving \eqref{eq:transfer}. The remaining identities follow
from $SZ=0$ and $Z^\top v=0$.
\end{proof}

For $a\in\R^k$ and $c\in\R$, write
\begin{equation}\label{eq:lift}
 u=Qa+cv.
\end{equation}
The identities above give
\begin{align}
 \norm{u}^2&=\norm{a}^2+c^2,\label{eq:liftnorm}\\
 u^\top(Y+S)u&=a^\top Ca+\mu c^2,\label{eq:liftenergy}\\
 \one^\top u&=0.\label{eq:liftcenter}
\end{align}
These are exact identities, with no approximation or spectral gap
assumption on $C$.

\section{Row-wise zero-curvature directions}\label{sec:rowdirections}

\begin{lemma}[A direction at every maximum-diagonal row]\label{lem:rowdirection}
Suppose $k>\ell$, \eqref{eq:stationaritymatrix} holds, and $S\ne0$.
For every $i\in I$ there exists a nonzero $a_i\in\R^k$ satisfying
\begin{equation}\label{eq:rowcertificate}
 Ca_i=ta_i,\qquad z_i^\top a_i=0,
 \qquad z_j^\top a_i=0\quad(j\notin I).
\end{equation}
Moreover, if $U_i=e_i a_i^\top$, then
\begin{equation}\label{eq:fullkernel}
 \mathcal H\operatorname{vec}(U_i)=0,
\end{equation}
where vectorization groups the coordinates by rows.
\end{lemma}
\begin{proof}
Fix $i\in I$ and use $Q,v,\mu$ from
Lemmas~\ref{lem:nonzero} and~\ref{lem:polar}. Consider the linear map
\begin{equation}\label{eq:linearmap}
 \begin{split}
 L_i:\R^k\times\R&\longrightarrow\R^\ell\times\R,\\
 (a,c)&\longmapsto\bigl(P^\top(Qa+cv),\ z_i^\top a\bigr).
 \end{split}
\end{equation}
Its domain has dimension $k+1>\ell+1$. By rank-nullity, its kernel
contains a nonzero pair $(a,c)$. Thus
\begin{equation}\label{eq:kernelconditions}
 P^\top u=0,
 \qquad z_i^\top a=0,
 \qquad (a,c)\ne(0,0),
 \qquad u=Qa+cv.
\end{equation}

Define $q=a^\top(C-tI_k)a$. Since $i\in I$ and $z_i^\top a=0$,
\eqref{eq:diagonalblock} gives
\begin{equation}\label{eq:localnonneg}
 q=a^\top\mathcal H_{ii}a\ge0.
\end{equation}
The first identity in \eqref{eq:kernelconditions} implies
$u^\top Gu=\norm{P^\top u}^2=0$.
Combining \eqref{eq:gram}, \eqref{eq:diagquadratic}, and
\eqref{eq:liftnorm}--\eqref{eq:liftcenter} gives the decomposition
\begin{equation}\label{eq:keydecomposition}
 \begin{split}
 0=u^\top Gu
 &=a^\top Ca+\mu c^2-\sum_j s_j u_j^2\\
 &=q+(\mu-t)c^2+\sum_j(t-s_j)u_j^2.
 \end{split}
\end{equation}
Each term in the last line is nonnegative, and $\mu-t>0$.
Consequently,
\begin{equation}\label{eq:vanishing}
 c=0,\qquad q=0,\qquad u_j=0\quad(j\notin I).
\end{equation}
The nonzero pair in \eqref{eq:kernelconditions} therefore has $a\ne0$.

Let $U=e_i a^\top$. Its full Hessian quadratic form is
\[
 \operatorname{vec}(U)^\top\mathcal H\operatorname{vec}(U)
 =a^\top\mathcal H_{ii}a=q=0.
\]
Because $\mathcal H\succeq0$, \eqref{eq:zerocurvature} yields
$\mathcal H\operatorname{vec}(U)=0$.
The $i$th block, together with $z_i^\top a=0$, implies
$(C-tI_k)a=0$. By \eqref{eq:transfer} and $c=0$,
\[
 Za=\sqrt t\,Qa=\sqrt t\,u.
\]
The support assertion in \eqref{eq:vanishing} now gives
$z_j^\top a=0$ for every $j\notin I$. Taking $a_i=a$ proves all the
claims.
\end{proof}

\section{Propagation through the off-diagonal Hessian blocks}\label{sec:propagation}

\begin{lemma}[Vanishing of a maximum-diagonal row]\label{lem:propagation}
Under the hypotheses of Lemma~\ref{lem:rowdirection}, every $i\in I$
satisfies
\[
 S_{ij}=0\qquad(j\ne i).
\]
\end{lemma}
\begin{proof}
For distinct $i,j$, define
\[
 d=z_i-z_j,\qquad M_{ij}=dd^\top+S_{ij}I_k.
\]
The $j$th block of \eqref{eq:fullkernel}, together with
\eqref{eq:crossblock}, gives
\begin{equation}\label{eq:crossnull}
 M_{ij}a_i=0\qquad(i\in I,\ j\ne i).
\end{equation}

Fix $i\in I$. If $j\notin I$, then \eqref{eq:rowcertificate}
implies $d^\top a_i=0$. Equation~\eqref{eq:crossnull} becomes
$S_{ij}a_i=0$. Since $a_i\ne0$, we obtain $S_{ij}=0$.

Suppose instead that $j\in I$ and $j\ne i$. Applying
\eqref{eq:crossnull} also at row $j$ gives
$M_{ij}a_i=M_{ij}a_j=0$; symmetry and
$(z_j-z_i)(z_j-z_i)^\top=dd^\top$ justify using the same matrix.
Assume $S_{ij}\ne0$. The equation
\[
 S_{ij}a_i=-d(d^\top a_i)
\]
then makes the nonzero vector $a_i$ parallel to the nonzero vector $d$.
The same reasoning applies to $a_j$. Since
$z_i^\top a_i=z_j^\top a_j=0$, it follows that
$z_i^\top d=z_j^\top d=0$. Thus
\[
 \norm{d}^2=(z_i-z_j)^\top d=0,
\]
contradicting $d\ne0$. Hence $S_{ij}=0$ in this case as well.
\end{proof}

\begin{proof}[Proof of Theorem~\ref{thm:main}]
The zero- and one-point cases are immediate. Center $P$ and $Z$ as
in \eqref{eq:centered}. At a second-order stationary point,
\eqref{eq:stationaritymatrix} holds, and
Lemma~\ref{lem:stresspsd} gives $S\succeq0$.
If $S=0$, \eqref{eq:rows} gives $R_{ij}=0$ for every pair, proving
the result.

Otherwise, $t>0$ and $I$ is nonempty by Lemma~\ref{lem:nonzero}.
Choose $i\in I$. Lemma~\ref{lem:propagation} and $S\one=0$ give
\[
 t=S_{ii}=-\sum_{j\ne i}S_{ij}=0,
\]
a contradiction. Therefore $S=0$ and $F_P(Z)=0$.
Undoing the centering does not change the objective or the distances.
Finally, $F_P$ is nonnegative everywhere, so a zero-residual point is a
global minimizer.
\end{proof}

\begin{remark}\label{rem:fullhessian}
The diagonal blocks \eqref{eq:diagonalblock} provide the nonnegative
term $q$ in \eqref{eq:keydecomposition}. They do not by themselves
suffice for the conclusion. The proof also uses that zero curvature of
a single-point perturbation forces that perturbation into the kernel of
the \emph{full} Hessian. The resulting off-diagonal equations
\eqref{eq:crossnull} are essential.
\end{remark}

\section{Formal statement and verification}\label{sec:formal}

The Lean source, verification records, and exact arithmetic scripts are
included as ancillary files with this paper. The Lean project proves
\nolinkurl{SStress.mainStatement_proved}. Its type is
\nolinkurl{SStress.MainStatement}, which quantifies over all natural numbers
$n,k,\ell$, all real configuration matrices $P,Z$, and the sole dimension
inequality $\ell<k$. The conclusion is zero objective at every
second-order critical point. Neither genericity nor a rank condition is
an extra formal hypothesis.

\subsection{Agreement with the mathematical problem}

The formal objective is
\[
 \frac14\sum_{i=1}^n\sum_{j=1}^n
 \left(\norm{z_i-z_j}^2-\norm{p_i-p_j}^2\right)^2.
\]
The theorem \nolinkurl{objective_eq_half_sum_lt} proves its equality to
\eqref{eq:objective}. The theorem
\nolinkurl{sqDist_eq_euclidean_norm_sq} identifies the implemented
coordinate dot product with the squared Euclidean norm.

The predicate \texttt{SecondOrderCritical} requires, for every real
perturbation matrix $U$,
\[
 \left.\frac{d}{dt}F_P(Z+tU)\right|_{t=0}=0,
 \qquad
 \left.\frac{d^2}{dt^2}F_P(Z+tU)\right|_{t=0}\ge0.
\]
The derivatives are Lean's actual real derivatives of the implemented
objective. The development proves that both derivatives exist at every
point of every such line; it does not rely on a default derivative value
at a nondifferentiable point. The gradient-stress and Hessian identities
are then derived from the quartic polynomial. Since the objective is a
polynomial, these directional conditions are precisely the usual
first- and second-order stationarity conditions used in
\eqref{eq:sosp}.

The final development also proves exact agreement of every measured
squared distance and global minimality in the separate corollaries
\begin{center}
\begin{tabular}{c}
\nolinkurl{exact_distances_of_secondOrderCritical}\\
\nolinkurl{global_minimum_of_secondOrderCritical}.
\end{tabular}
\end{center}
The strict-saddle consequence in Corollary~\ref{cor:landscape} is stated
here as an immediate mathematical consequence of the main theorem;
it is not represented as a separately named formal theorem in the archive.

\subsection{Proof organization}

\begin{center}
\small
\begin{tabular}{@{}>{\raggedright\arraybackslash}p{0.35\linewidth}>{\raggedright\arraybackslash}p{0.57\linewidth}@{}}
\toprule
Lean module & Mathematical content \\
\midrule
\texttt{Definitions}, \texttt{ObjectiveSemantics}
 & Objective, Euclidean distances, criticality, and zero-residual semantics \\
\texttt{Algebra}, \texttt{Derivatives}
 & Centering, Gram identity, and derivatives of the original quartic \\
\texttt{StressPositive}
 & Stress positivity and the high stress eigenvector \\
\texttt{PolarNormalization}
 & Positive square root and the identities in Lemma~\ref{lem:polar} \\
\texttt{RowDirectionsBasic}, \texttt{RowDirections}
 & Rank-nullity, nonnegative decomposition, and row certificates \\
\texttt{KernelPropagation}
 & Zero-curvature propagation and the stress-row contradiction \\
\texttt{Main}, \texttt{Audit}
 & Universal theorem, semantic corollaries, and transitive axiom audit \\
\bottomrule
\end{tabular}
\end{center}

\subsection{Recorded kernel check and reproducibility}

The development is pinned to Lean~4.19.0 and mathlib commit
\begin{center}
\texttt{c44e0c8ee63ca166450922a373c7409c5d26b00b}.
\end{center}
The dependency manifest fixes the transitive package revisions.
The recorded verification run on 21 September 2026 completed both of the
following commands with exit status zero:
\begin{verbatim}
lake build SStress.Audit
lake env lean --trust=0 SStress/Audit.lean
\end{verbatim}
The audit checks an explicit, expanded type for the public theorem and
prints the objective and criticality definitions for inspection. It also
collects the transitive axioms of every theorem declaration in the
\texttt{SStress} namespace. The recorded run checked 168 theorem declarations,
including helper declarations, and allowed only
\begin{center}
\texttt{propext},\qquad
\texttt{Classical.choice},\qquad
\texttt{Quot.sound}.
\end{center}
No problem-specific axiom, proof placeholder, or native-evaluation axiom
is used. These are the standard propositional extensionality, classical
choice, and quotient soundness axioms of the underlying Lean development.

From the ancillary directory \nolinkurl{anc/lean4/}, after installing
the pinned Lean toolchain and obtaining the manifest-pinned dependencies,
the verification can be repeated with
\begin{verbatim}
python3 scripts/verify_release.py
\end{verbatim}
The script checks dependency revisions, scans the project source, runs
the build and zero-trust kernel check, audits axioms, and records source
hashes, command logs, and exit statuses. The archived JSON summary has
\nolinkurl{overall_status} equal to \texttt{PASSED}.
Kernel verification establishes the encoded theorem relative to the
stated foundational axioms and the Lean implementation. The semantic
identifications above make explicit which mathematical problem is
encoded; they also make the scope of the formal claim independently
reviewable.

\section{Exact arithmetic checks}\label{sec:checks}

The supplementary script \nolinkurl{sstress_proof_identity_checks.py}
uses Python's standard-library rational arithmetic. It checks the
centered Gram identity, the gradient, all Hessian blocks, and first and
second directional derivatives obtained independently by expanding the
edge quartics. These finite checks are supplementary tests of formulas
and control configurations. They are not used to prove the universal
theorem.

One control has $n=7$ and $k=\ell=2$. Take
\[
 P=\begin{pmatrix}
 -2&0\\-1&0\\0&0\\1&0\\2&0\\0&1\\0&-1
 \end{pmatrix},
 \qquad
 Z=\begin{pmatrix}
 -2&0\\-1&0\\0&0\\1&0\\2&0\\0&1\\0&1
 \end{pmatrix}.
\]
The script centers both configurations and verifies
$F_P(Z)=8$, $\nabla F_P(Z)=0$, and
$\nabla^2F_P(Z)\succeq0$. The last assertion is certified by a
positive-definite $11\times11$ principal block with an exact $LDL^\top$
factorization and a zero complementary Schur complement. Thus a
positive-objective second-order stationary point is possible when the
strict dimension inequality is removed.

A second control has $n=5$, $\ell=1$, and $k=2$, with
\[
 P=\sqrt{13}\begin{pmatrix}1\\1\\-1\\-1\\0\end{pmatrix},
 \qquad
 Z=\begin{pmatrix}3&0\\-3&0\\0&3\\0&-3\\0&0\end{pmatrix}.
\]
Its true Gram matrix has rational entries, so every calculation is
performed exactly. The verified values are
\[
 C=18I_2,\qquad
 (s_1,s_2,s_3,s_4,s_5)=(18,18,18,18,8),\qquad
 F_P(Z)=3640.
\]
The gradient vanishes and every single-point diagonal Hessian block is
positive semidefinite. Nevertheless, the perturbation
\[
 U=\begin{pmatrix}0&1\\0&1\\0&0\\0&0\\0&0\end{pmatrix}
\]
has
\[
 D^2F_P(Z)[U,U]=-144.
\]
This control verifies the distinction in
Remark~\ref{rem:fullhessian}: positive semidefiniteness of the diagonal
blocks does not replace positive semidefiniteness of the full Hessian.

\Needspace{8\baselineskip}
\section{Scope and further questions}\label{sec:scope}

Theorem~\ref{thm:main} establishes the conjectured qualitative landscape
for the complete, exact, equal-weight objective. The proof uses the
centered identity \eqref{eq:gram} and the common covariance term in every
block \eqref{eq:diagonalblock}. Missing edges or unequal weights change
these structures. The theorem therefore makes no assertion for those
measurement models, for anchor constraints, or for noisy data.

A quantitative extension would require information beyond the existence
of a negative-curvature direction. In particular,
\eqref{eq:keydecomposition} exploits exact vanishing, and the propagation
step uses the exact kernel implication \eqref{eq:zerocurvature}.
Obtaining useful error bounds for approximate stationarity, or stability
under perturbations of the distances, remains a separate problem.

\end{document}